\pdfoutput=1
\documentclass{scrartcl}
\usepackage{amsthm,amssymb,mathtools,bm,ifpdf,authblk}
\usepackage{url}
\usepackage[abbrev]{amsrefs}
\usepackage[T1]{fontenc}
\usepackage{fix-cm}
\usepackage{microtype}
\usepackage{tocstyle}
\usepackage{hyperref}

\newtheorem{thm}{Theorem}[section]
\newtheorem*{thm*}{Theorem}
\newtheorem{lemma}[thm]{Lemma}
\newtheorem*{lemma*}{Lemma}

\newtheorem{cor}[thm]{Corollary}
\newtheorem*{cor*}{Corollary}

\theoremstyle{definition}

\newtheorem{ex}[thm]{Example}

\newtheorem{rem}[thm]{Remark}

\newcommand{\GG}{\ensuremath{\mathbf{G}}}
\newcommand{\G}{\ensuremath{\mathsf{G}}}
\newcommand{\A}{\ensuremath{\mathsf{A}}}
\newcommand{\irr}{\ensuremath{\mathrm{irr}}}
\newcommand{\wirr}{\ensuremath{\widetilde{\mathrm{irr}}}}

\newcommand{\fg}{\ensuremath{\mathfrak g}}
\newcommand{\Places}{\ensuremath{\mathcal V}}
\newcommand{\RF}{\ensuremath{\mathfrak K}}

\DeclareMathOperator{\Cl}{Cl}
\DeclareMathOperator{\sign}{sgn}

\numberwithin{equation}{section}
\allowdisplaybreaks[1]

\renewcommand{\le}{\leqslant}
\renewcommand{\ge}{\geqslant}

\def\emph{}
\DeclareTextFontCommand{\bfemph}{\bf}
\DeclareTextFontCommand{\itemph}{\it}
\def\emph{\bfemph}

\makeatletter
\def\blankfootnote{\xdef\@thefnmark{}\@footnotetext}

\newcommand*{\textlabel}[2]{%
  \edef\@currentlabel{#1}% Set target label
  \phantomsection% Correct hyper reference link
  #1\label{#2}% Print and store label
}
\makeatother

\newcommand{\idx}[1]{\lvert#1\rvert}

\newcommand{\incl}{\hookrightarrow}
\newcommand{\acts}{\ensuremath{\curvearrowright}}

\let\AA\undefined
\newcommand{\AA}{\mathbf{A}}
\newcommand{\FF}{\mathbf{F}}
\newcommand{\QQ}{\mathbf{Q}}
\newcommand{\NN}{\mathbf{N}}
\newcommand{\ZZ}{\mathbf{Z}}
\newcommand{\CC}{\mathbf{C}}
\newcommand{\MM}{\mathbf M}

\newcommand{\Zeta}{\ensuremath{\mathsf{Z}}}
\newcommand{\xx}{\ensuremath{\bm x}}

\newcommand{\fp}{\mathfrak{p}}
\newcommand{\fP}{\mathfrak{P}}
\newcommand{\fo}{\mathfrak{o}}
\newcommand{\fO}{\mathfrak{O}}

\newcommand{\cA}{\mathcal{A}}
\newcommand{\cE}{\mathcal{E}}
\newcommand{\cH}{\mathcal{H}}

\DeclareMathOperator{\GL}{GL}
\DeclareMathOperator{\End}{End}
\DeclareMathOperator{\Spec}{Spec}
\DeclareMathOperator{\topo}{top}
\DeclareMathOperator{\trace}{trace}
\newcommand{\Euler}{\ensuremath{\chi}}

\DeclareMathOperator{\dd}{d\!}
\newcommand{\normal}{\triangleleft}
\newcommand{\dtimes}{\ensuremath{\,\cdotp}}
\newcommand{\noof}[1]{\#{#1}}

\DeclarePairedDelimiter{\abs}{\lvert}{\rvert}
\DeclarePairedDelimiter{\norm}{\lVert}{\rVert}
\DeclarePairedDelimiter{\red}{\lfloor}{\rfloor}

\newcommand{\llb}{\ensuremath{[\![ }}
\newcommand{\llp}{\ensuremath{(\!( }}
\newcommand{\rrb}{\ensuremath{]\!] }}
\newcommand{\rrp}{\ensuremath{)\!) }}

\makeatletter
\newcommand{\subalign}[1]{%
  \vcenter{%
    \Let@ \restore@math@cr \default@tag
    \baselineskip\fontdimen10 \scriptfont\tw@
    \advance\baselineskip\fontdimen12 \scriptfont\tw@
    \lineskip\thr@@\fontdimen8 \scriptfont\thr@@
    \lineskiplimit\lineskip
    \ialign{\hfil$\m@th\scriptstyle##$&$\m@th\scriptstyle{}##$\crcr
      #1\crcr
    }%
  }
}
\makeatother
\ifpdf
  \hypersetup{pdftitle={Stability results for local zeta functions of groups and related structures},
    pdfauthor={Tobias Rossmann}
  }
\fi
\title{Stability results for local zeta functions of groups and related structures}
\author{Tobias Rossmann}
\affil{\small Fakult\"at f\"ur Mathematik, Universit\"at Bielefeld, D-33501
  Bielefeld, Germany}
\date{April 2015}
\begin{document}

\maketitle
\thispagestyle{empty}

\begin{abstract}
  \small
  Various types of local zeta functions studied in asymptotic group
  theory admit two natural operations: (1) change the prime and (2) perform
  local base extensions.
  Often, the effects of both of these operations
  can be expressed simultaneously in terms of a single explicit formula.
  We show that in these cases, the behaviour of local zeta functions under
  variation of the prime in a set of density~$1$
  in fact completely determines these functions for almost all primes 
  and, moreover, it also determines their behaviour under local base extensions.
  We discuss applications to topological zeta functions, functional equations,
  and questions of uniformity.
\end{abstract}

\blankfootnote{\indent{\itshape 2000 Mathematics Subject Classification.}
 11M41, 20F69.


 This work is supported by the DFG Priority Programme  ``Algorithmic and
 Experimental Methods in Algebra, Geometry and Number Theory'' (SPP 1489).}

%%%%%%%%%%%%%%%%%%%%%%%%%%%%%%%%%%%%%%%%%%%%%%%%%%%%%%%%%%%%%%%%%%%%%%%%
\section{Introduction}
%%%%%%%%%%%%%%%%%%%%%%%%%%%%%%%%%%%%%%%%%%%%%%%%%%%%%%%%%%%%%%%%%%%%%%%%

For a finitely generated nilpotent group $G$ and
a prime $p$, let $\zeta^{\wirr}_{G,p}(s)$ be the Dirichlet series enumerating
continuous irreducible finite-dimensional complex representations of 
the pro-$p$ completion $\hat G_p$ of $G$, counted up to equivalence and
tensoring with continuous $1$-dimensional representations.
In this article, we prove statements of the following form.

\begin{thm*}
  Let $G$ and $H$ be finitely generated nilpotent groups
  such that $\zeta^{\wirr}_{G,p}(s) = \zeta^{\wirr}_{H,p}(s)$ for all primes $p$ in a set of
  density $1$.
  Then $\zeta^{\wirr}_{G,p}(s) = \zeta^{\wirr}_{H,p}(s)$ for almost all primes $p$.
\end{thm*}

Much more can be said if, in addition to the variation of $p$, we also take into
account local base extensions.
To that end, as explained in~\cite{SV14}, we may
construct a unipotent group scheme $\G$ over $\ZZ$ such that $\G\otimes_{\ZZ} \QQ$ is
the Mal'cev completion of $G$ regarded as an algebraic $\QQ$-group.
It follows that  $G$ and $\G(\ZZ)$ are
commensurable and that $\hat G_p = \G(\ZZ_p)$ for almost all $p$, where $\ZZ_p$
denotes the $p$-adic integers;
let $\mathsf H$ be a unipotent group scheme associated with $H$ in the same way.
The preceding theorem can be sharpened as follows.
\begin{thm*}
  Suppose that $\zeta^{\wirr}_{\G(\ZZ_p)}(s) = \zeta^{\wirr}_{\mathsf H(\ZZ_p)}(s)$ for all primes $p$ in a set of
  density $1$.
  Then $\zeta^{\wirr}_{\G(\fO_K)}(s) = \zeta^{\wirr}_{\mathsf H(\fO_K)}(s)$
  for almost all primes $p$ and all finite extensions $K$ of the field $\QQ_p$
  of $p$-adic numbers, where
  $\fO_K$ denotes the valuation ring of $K$. 
\end{thm*}

Explicit formulae in the spirit of Denef's work on Igusa's local zeta
function (see~\cite{Den91a}) have been previously obtained 
for zeta functions such as $\zeta_{\G(\fO_K)}^{\wirr}(s)$ (see~\cite{SV14}).
These formulae are well-behaved under both variation of the
prime $p$ and under local base extensions $K/\QQ_p$.
Our technical main result, Theorem~\ref{thm:rigidity}, shows that if
some explicit formula behaves well under these two operations, then
every formula which remains valid if $p$ is changed also necessarily 
remains valid if local base extensions are performed,
at least after excluding a finite number of primes.
We prove this by interpreting the geometric ingredients of the explicit formulae
under consideration in terms of $\ell$-adic Galois representations and by
invoking Chebotarev's density theorem.
A collection of arithmetic applications of these techniques is given in
\cite{Ser12} which provides  the main inspiration for the present article.
In particular, Theorem~\ref{thm:rigidity} below draws upon
the following result.

\begin{thm*}[{\cite[Thm~1.3]{Ser12}}]
  Let $V$ and $W$ be schemes of finite type over $\ZZ$.
  Suppose that $\noof{V(\FF_p)} = \noof{W(\FF_p)}$ for all $p$ in a set of
  primes of density $1$.
  Then $\noof{V(\FF_{p^f})} = \noof{W(\FF_{p^f})}$
  for almost all primes $p$ and all $f\in \NN$.
\end{thm*}

As a first application, in \S\ref{s:top}, we consider consequences of the
local results obtained in this article to topological zeta
functions---the latter zeta functions arise as limits ``$p \to 1$'' of local
ones.
In particular, we provide a rigorous  justification for the process of deriving
topological zeta functions from suitably uniform $p$-adic ones.

Thanks to \cite{Vol10} and its subsequent applications in \cite{AKOV13,SV14},
various types of local zeta functions arising in asymptotic group and ring
theory are known to generically satisfy functional equations
under ``inversion of $p$'', an operation defined on the level of carefully
chosen explicit formulae.
As another application of the results in this article,
we show in~\S\ref{s:feqn} that such local functional equations are independent
of the chosen formulae whence they admit the expected interpretation for uniform
examples in the sense of~\cite[\S 1.2.4]{dSW08}.

\subsection*{Acknowledgement}
I would like to thank Christopher Voll for helpful discussions.

\subsection*{\textnormal{\textit{Notation}}}
We write $\NN = \{ 1,2,\dotsc\}$ and use the symbol ``$\subset$'' to indicate
not necessarily proper inclusion.
Throughout,  $k$ is always a number field with ring of integers $\fo$.
We write~$\Places_k$ for the set of non-Archimedean places of $k$.
Given $v\in \Places_k$, let $k_v$ be the $v$-adic completion of~$k$
and let $\RF_v$ be its residue field.
We let $q_v$ and $p_v$ denote the cardinality and characteristic of $\RF_v$, respectively.
Choose an algebraic closure $\bar k_v$ of $k_v$ and for $f\ge 1$, let
$k_v^{(f)} \subset \bar k_v$ denote the unramified extension of degree $f$ of $k_v$.
The residue field $\bar\RF_v$ of $\bar k_v$ is an algebraic closure of $\RF_v$.
Having fixed $\bar k_v$, we identify the
residue field $\RF_v^{(f)}$ of $k_v^{(f)}$ with the extension of degree
$f$ of $\RF_v$ within $\bar\RF_v$.
We let $\fO_K$ denote the valuation ring of a non-Archimedean local field $K$,
and we let $\fP_K$ denote the maximal ideal of $\fO_K$.
We write $q_K = \noof{(\fO_K/\fP_K)}$.

%%%%%%%%%%%%%%%%%%%%%%%%%%%%%%%%%%%%%%%%%%%%%%%%%%%%%%%%%%%%%%%%%%%%%%%%
\section{Local zeta functions of groups, algebras, and modules}
\label{s:local_zetas}
%%%%%%%%%%%%%%%%%%%%%%%%%%%%%%%%%%%%%%%%%%%%%%%%%%%%%%%%%%%%%%%%%%%%%%%%

We recall the definitions of the representation and subobject zeta functions
considered in this article on a formal level which disregards questions of
convergence (and finiteness).

\paragraph{Representation zeta functions.}
For further details on the following, we refer the reader to \cites{Klo13} and
\cite[\S 4]{Vol14}.
Let $G$ be a topological group. For $n\in \NN$, let $r_n(G)\in \NN \cup
\{0,\infty\}$ denote the number of equivalence classes of continuous irreducible
representations $G \to \GL_n(\CC)$.
The \emph{representation zeta function} of $G$ is
$\zeta_{G}^\irr(s) = \sum_{n=1}^\infty r_n(G)n^{-s}$.
Two continuous complex representations $\varrho$ and $\sigma$ of $G$ are
\emph{twist-equivalent} if $\varrho$ is equivalent to $\sigma\otimes_{\CC} \alpha$ for
a continuous $1$-dimensional complex representation $\alpha$ of $G$.
For $n\in \NN$, let $\tilde r_n(G)$ denote the number of twist-equivalence
classes of continuous irreducible representations $G \to \GL_n(\CC)$.
The \emph{twist-representation zeta function} of $G$ is 
$\zeta_{G}^{\wirr}(s) = \sum_{n=1}^\infty \tilde r_n(G)n^{-s}$.

\paragraph{Subobject zeta functions.}
The following types of zeta functions go back to \cites{GSS88,Sol77}, see
\cite{Vol14} for a recent survey;
we use the formalism from \cite[\S 2.1]{topzeta}.
Let $R$ be a commutative, unital, and associative ring.
Let $\A$ be a possibly non-associative $R$-algebra.
For $n\in \NN$, let $a_n^\le(\A)$ denote the number of
subalgebras $\mathsf U$ of $\A$ such that the $R$-module quotient $\mathsf
A/\mathsf U$ has cardinality $n$.
The \emph{subalgebra zeta function} of $\A$ is
$\zeta_{\A}^\le(s) = \sum_{n=1}^\infty a_n^{\le}(\A) n^{-s}$.
Similarly, let $\mathsf M$ be an $R$-module and let $\mathsf E$ be an
associative $R$-subalgebra of $\End_R(\mathsf M)$.
For $n\in \NN$, let $b_n(\mathsf E\acts \mathsf M)$ denote the number of
$(\mathsf E +  R 1_{\mathsf M})$-submodules $\mathsf U \subset \mathsf M$ 
such that the $R$-module quotient $\mathsf
M/\mathsf U$ has cardinality $n$.
The \emph{submodule zeta function} of $\mathsf E$ acting on $\mathsf M$ is
$\zeta_{\mathsf E \acts \mathsf M}(s) = \sum_{n=1}^\infty b_n(\mathsf E\acts
\mathsf M) n^{-s}$.
An important special case of a submodule zeta function is the \emph{ideal zeta function} 
of an $R$-algebra $\A$, defined by enumerating ($R$-)ideals of $\mathsf
A$ with finite $R$-module quotients.

A group-theoretic motivation for studying these zeta functions is given
in \cite[\S 4]{GSS88}: given a finitely generated nilpotent group
$G$, there exists a Lie ring $\mathsf L$ such that for almost all primes $p$, the
subalgebra (resp.~ideal) zeta function of $\mathsf L\otimes_{\ZZ}\ZZ_p$ 
enumerates precisely the subgroups (resp.~normal subgroups) of finite index of
$\hat G_p$, the pro-$p$ completion of $G$. 

%%%%%%%%%%%%%%%%%%%%%%%%%%%%%%%%%%%%%%%%%%%%%%%%%%%%%%%%%%%%%%%%%%%%%%%%
\section{Stability under base extension for local maps of Denef type}
\label{s:local_maps}
%%%%%%%%%%%%%%%%%%%%%%%%%%%%%%%%%%%%%%%%%%%%%%%%%%%%%%%%%%%%%%%%%%%%%%%%

Given a possibly non-associative $\ZZ$-algebra $\mathsf A$, the underlying
$\ZZ$-module of which is finitely generated, du~Sautoy and Grunewald
\cite{dSG00} established the existence of schemes $V_1,\dotsc,V_r$ and rational
functions $W_1,\dotsc,W_r \in \QQ(X,Y)$ such that 
for all primes $p\gg 0$, 
\begin{equation}
  \label{eq:denef}
  \zeta_{\mathsf A\otimes\ZZ_p}^{\le}(s) = \sum_{i=1}^r \noof{V_i(\FF_p)}\dtimes W_i(p,p^{-s}).
\end{equation}
In this section, we shall concern ourselves with formulae
of the same shape as \eqref{eq:denef}.
We will focus on the behaviour of such formulae under variation of $p$ as
well as under base extension.
As we will see, while the $V_i$ and $W_i$ in \eqref{eq:denef} are in no way
uniquely determined, valuable relations can be deduced from the validity 
of equations such as \eqref{eq:denef} alone.
To that end, we develop a formalism of ``local maps'' which, as we will
explain, specialises to various group-theoretic instances of local zeta functions.

%%%%%%%%%%%%%%%%%%%%%%%%%%%%%%%%%%%%%%%%%%%%%%%%%%%%%%%%%%%%%%%%%%%%%%%%
\subsection{Local maps of Denef type}
\label{ss:denef_type}
%%%%%%%%%%%%%%%%%%%%%%%%%%%%%%%%%%%%%%%%%%%%%%%%%%%%%%%%%%%%%%%%%%%%%%%%

The formalism developed in the following is closely related to the ``systems of
local zeta functions'' considered in \cite[\S 5.2]{topzeta}.
By a \emph{$k$-local map} in $m$ variables we mean a map
\[
\Zeta\colon \Places_k\setminus S_\Zeta \times \NN \to \QQ(Y_1,\dotsc,Y_m),
\]
where $S_\Zeta \subset \Places_k$ is finite;
in the following, the number $m$ will be fixed.

Local maps provide a convenient formalism for studying families of local zeta
functions as follows:
for $v\in \Places_k\setminus S_\Zeta$ and $f\in \NN$,
let $\hat\Zeta(v,f)$ denote the meromorphic function
$\Zeta(v,f)(q_v^{-fs_1},\dotsc,q_v^{-fs_m})$ in complex variables
$s_1,\dotsc,s_m$.
Note that given $(v,f)$, the functions $\Zeta(v,f)$ and $\hat\Zeta(v,f)$ determine each other.
Let $K$ be a non-Archimedean local field endowed with an embedding
$k\subset K$.
We may regard $K$ as a finite extension of $k_v$ for a unique $v\in
\Places_k$.
Let $f$ be the inertia degree of $K/k_v$.
Given a $k$-local map $\Zeta$,
if $v\not\in S_\Zeta$,
write $\Zeta_K := \Zeta(v,f)$ and $\hat\Zeta_K(s_1,\dotsc,s_m) := \hat\Zeta(v,f)$.

We say that two $k$-local maps $\Zeta$ and $\Zeta'$
are \emph{equivalent} if they coincide on $\Places_k \setminus S \times \NN$
for some finite $S \supset S_\Zeta \cup S_{\Zeta'}$.
We are usually only interested in local maps up to equivalence.
Let $V$ be a separated $\fo$-scheme of finite type and let $W \in
\QQ(X,Y_1,\dotsc,Y_m)$ be regular at $(q,Y_1,\dotsc,Y_m)$ for each integer $q > 1$. 
Define a $k$-local map
\[
[V\dtimes W]\colon \Places_k\times\NN\to\QQ(Y_1,\dotsc,Y_m),\quad
(v,f) \mapsto \noof V(\RF_v^{(f)}) \dtimes W(q_v^f,Y_1,\dotsc,Y_m).
\]
We say that a $k$-local map is of \emph{Denef type} if it is equivalent to a
finite (pointwise) sum of maps of the form $[V\dtimes W]$;
for a motivation of our terminology, see \S\ref{sss:igusa}.
Note that local maps of Denef type contain further information compared with
equations such as~\eqref{eq:denef}.
Namely, in addition to the variation of the prime, they also take into
account local base extensions.

%%%%%%%%%%%%%%%%%%%%%%%%%%%%%%%%%%%%%%%%%%%%%%%%%%%%%%%%%%%%%%%%%%%%%%%%
\subsection{Main examples of local maps}
\label{ss:denef_examples}
%%%%%%%%%%%%%%%%%%%%%%%%%%%%%%%%%%%%%%%%%%%%%%%%%%%%%%%%%%%%%%%%%%%%%%%%

We discuss the primary examples of $k$-local maps of Denef type of interest to
us.
These local maps will be constructed from an $\fo$-form
of a $k$-object and only be defined up to equivalence.
In the univariate case $m=1$, we simply write $Y = Y_1$ and $s = s_1$.
We also allow $m > 1$ since various types of univariate zeta functions in the
literature are most conveniently expressed and analysed as specialisations of multivariate
$p$-adic integrals.

\subsubsection{Generalised Igusa zeta functions \textnormal{(\cites{Den87,VZG08}; cf.~\cite[Ex.~5.11(ii),(vi)]{topzeta})}}
\label{sss:igusa}
Let $a_1,\dotsc,a_m \normal k[X_1,\dotsc,X_n]$ be non-zero ideals.
Write $\mathfrak a_i = a_i \cap \fo[X_1,\dotsc,X_n]$.
Established results from $p$-adic integration pioneered by Denef
show that there exists a $k$-local map
$\Zeta^{a_1,\dotsc,a_m}\colon \Places_k\setminus S \times \NN \to
\QQ(Y_1,\dotsc,Y_m)$ of Denef type such that
for all $v\in \Places_k\setminus S$, all finite extensions $K/k_v$, 
and all $s_1,\dotsc,s_m\in\CC$ with $\mathrm{Re}(s_1),\dotsc,\mathrm{Re}(s_m)\ge 0$,
we have
\[
\hat \Zeta^{a_1,\dotsc,a_m}_K(s_1,\dotsc,s_m) =
\int_{\fO_K^n} \norm{\mathfrak a_1(\xx)}^{s_1} \dotsb \norm{\mathfrak
  a_m(\xx)}^{s_m} \dd\mu_K(\xx),
\]
where $\mu_K$ denotes the normalised Haar measure on $K^n$
and $\norm{\dtimes}$ the usual maximum norm.
Note that the equivalence class of $\Zeta^{a_1,\dotsc,a_m}$ remains unchanged 
if we replace $\mathfrak a_i$ by another $\fo$-ideal with $k$-span $a_i$.

\subsubsection{Subalgebra and submodule zeta functions
  \textnormal{(\cite{dSG00}; cf.~\cite[Ex.~5.11(iii)]{topzeta})}}
\label{sss:subobjects}

Let $\cA$ be a not necessarily associative finite-dimensional $k$-algebra.
Choose any $\fo$-form $\A$ of $\cA$ which is finitely generated as an $\fo$-module.
By~\cite[Thm~1.4]{dSG00} (cf.~\cite[Thm~5.16]{topzeta}), 
there exists a $k$-local map $\Zeta^{\cA}\colon \Places_k\setminus S\times \NN\to\QQ(Y)$
of Denef type characterised
by
\[
\hat\Zeta^{\cA}_K(s) = \zeta_{\A\otimes_{\fo}\fO_K}^{\le}(s),
\]
where we regard $\A\otimes_{\fo}\fO_K$ as an $\fO_K$-algebra (so that we
enumerate $\fO_K$-subalgebras of finite index).
The equivalence class of $\Zeta^{\cA}$ only depends on $\cA$ and not
on $\A$.

\begin{rem}
  For a finite extension $K/k_v$, instead of regarding $\A \otimes_{\fo} \fO_K$
  as an $\fO_K$-algebra, we could also consider it as an algebra over $\fo_v$
  and enumerate its $\fo_v$-subalgebras or $\fo_v$-ideals.
  While the latter type of ``extension of scalars''~\cite[p.~188]{GSS88} 
  has the advantage of admitting a group-theoretic interpretation for nilpotent
  Lie rings, its effect on zeta functions is only understood in very few cases,
  see~\cite[Thm~3--4]{GSS88} and~\cites{SV15a,SV15b}.
\end{rem}

Let $M$ be a finite-dimensional vector space over $k$, and let $\cE$
be an associative subalgebra of $\End_k(M)$.
Choose an $\fo$-form $\mathsf M$ 
of $M$ which is finitely generated as an $\fo$-module
and an $\fo$-form $\mathsf E\subset \End_{\fo}(\mathsf M)$ of $\cE$.
Similarly to the case of algebras from above,
we obtain a $k$-local map $\Zeta^{\cE\acts M}\colon \Places_k\setminus S\times
\NN\to\QQ(Y)$ of Denef type with
\[
\hat\Zeta^{\cE\acts M}_K(s) = \zeta_{(\mathsf E\otimes_{\fo} \fO_K) \acts (\mathsf M\otimes_{\fo}\fO_K)}(s).
\]

\subsubsection{Representation zeta functions: unipotent groups \textnormal{(\cite{SV14})}}
\label{sss:unipotent}

Let $\GG$ be a unipotent algebraic group over $k$.
Choose an affine group scheme $\G$ of finite type over $\fo$ with $\G
\otimes_{\fo} k\approx_k \GG$.
There exists a finite set $S\subset\Places_k$ such that
$\G(\fO_K)$ is a finitely generated nilpotent pro-$p_v$ group for $v\in
\Places_k\setminus S$ and all finite extensions $K/k_v$.
By \cite[Pf of Thm~A]{SV14}, after enlarging $S$,
we obtain a $k$-local map $\Zeta^{\GG,\wirr}\colon \Places_k\setminus S\times
\NN\to\QQ(Y)$ of Denef type characterised by
$$\hat\Zeta^{\GG,\wirr}_K(s) = \zeta_{\G(\fO_K)}^{\wirr}(s);$$
the equivalence class of $\Zeta^{\GG,\wirr}$ only depends on $\GG$ and not on $\G$.

\subsubsection{Representation zeta functions: FAb $p$-adic analytic pro-$p$
  groups \textnormal{(\cite{AKOV13})}}
\label{sss:perfect}
Let $\bm\fg$ be a finite-dimensional perfect Lie $k$-algebra.
Choose an $\fo$-form $\fg$ of $\bm\fg$ which is finitely generated as an $\fo$-module.
As explained in \cite[\S 2.1]{AKOV13}, there exists a finite set
$S\subset\Places_k$ (including all the places which ramify over $\QQ$) such that
for each $v\in \Places_k\setminus S$ and $f\in \NN$, the space
$\G^1(\fo_v^{(f)}) := \fp_v^{(f)} (\fg \otimes_{\fo} \fo_v^{(f)})$ can be naturally endowed
with the structure of a FAb $p_v$-adic analytic pro-$p_v$ group;
here, $\fo_v^{(f)}$ denotes the valuation ring of $k_v^{(f)}$
and $\fp_v^{(f)}$ the maximal ideal of $\fo_v^{(f)}$.
By~\cite[\S\S 3--4]{AKOV13}, after further enlarging $S$, we obtain a map of Denef type
$\Zeta^{\bm\fg,\irr}\colon \Places_k\setminus S\times \NN\to\QQ(Y)$
with $$\hat\Zeta^{\bm\fg,\irr}_{k_v^{(f)}}(s)
=\zeta^{\irr}_{\G^1(\fo_v^{(f)})}(s),$$
the equivalence class of which only depends on $\bm \fg$.
Note that in contrast to the preceding examples, here we restrict attention
to (absolutely) unramified extensions $k_v^{(f)}/\QQ_{p_v}$;
by considering ``higher congruence subgroups'' $\G^m(\dtimes)$ as in
\cite[Thm~A]{AKOV13}, this restriction could be removed at the cost of a more
technical exposition.

%%%%%%%%%%%%%%%%%%%%%%%%%%%%%%%%%%%%%%%%%%%%%%%%%%%%%%%%%%%%%%%%%%%%%%%%
\subsection{Main results}
\label{ss:main}
%%%%%%%%%%%%%%%%%%%%%%%%%%%%%%%%%%%%%%%%%%%%%%%%%%%%%%%%%%%%%%%%%%%%%%%%

The following result, which we will prove in \S\ref{ss:pf_rigidity}, constitutes the
technical heart of this article.
Throughout, by the \emph{density} of a set of places or primes, we mean the
natural one as in \cite[\S 3.1.3]{Ser12}. 
As before, let $k$ be a number field with ring of integers $\fo$.

\begin{thm}
  \label{thm:rigidity}
  Let $V_1,\dotsc,V_r$ be separated $\fo$-schemes of finite type and let $W_1,\dotsc,W_r\in
  \QQ(X,Y_1,\dotsc,Y_m)$.
  Suppose that $(q,Y_1,\dotsc,Y_m)$ is a regular point of each
  $W_i$ for each integer $q > 1$.
  Let $P\subset \Places_k$ be a set of places of density $1$ 
  and suppose that   for all $v \in P$,
  $$\sum\limits_{i=1}^r \noof{V_i(\RF_v)} \dtimes W_i(q_v,Y_1,\dotsc,Y_m) = 0.$$
  Then there exists a finite set $S\subset \Places_k$ such that
  for all $v \in \Places_k\setminus S$ and all $f\in \NN$,
  $$\sum\limits_{i=1}^r \noof{V_i(\RF_v^{(f)})} \dtimes W_i(q_v^f,Y_1,\dotsc,Y_m) = 0.$$
\end{thm}

We now discuss important consequences of Theorem~\ref{thm:rigidity} for
local maps of Denef type.
The following implies the first two theorems stated in the introduction.

\begin{cor}
  \label{cor:rigidity}
  Let $\Zeta,\Zeta'$ be $k$-local maps of Denef type.
  Let $P\subset \Places_k\setminus (S_{\Zeta}\cup S_{\Zeta'})$ have density $1$
  and let $\Zeta(v,1) = \Zeta'(v,1)$ for all $v\in P$.
  Then $\Zeta$ and $\Zeta'$ are equivalent.
  That is, there exists a finite $S\supset S_\Zeta \cup S_{\Zeta'}$ with
  $\Zeta(v,f) = \Zeta'(v,f)$ for all $v\in \Places_k\setminus S$ and $f\in \NN$.
\end{cor}
\begin{proof}
  Apply Theorem~\ref{thm:rigidity} to the difference $\Zeta-\Zeta'$.
\end{proof}

\begin{ex}
Let
\[
\mathsf H(R) =
\begin{bmatrix}
  1 & R & R \\
    & 1 & R \\
    &   & 1
\end{bmatrix}
\le\GL_3(R)
\]
be the ``natural'' $\ZZ$-form of the Heisenberg group.
It has long been known \cite{NM89} that the global twist-representation zeta
function of $\mathsf H(\ZZ)$
satisfies $\zeta_{\mathsf H(\ZZ)}^{\wirr}(s) = \zeta(s-1)/\zeta(s)$.
For $\idx{k:\QQ}=2$, 
Ezzat \cite[Thm~1.1]{Ezz14} found that
$\zeta^{\wirr}_{\mathsf H(\fo)}(s) = \zeta_k(s-1)/\zeta_k(s)=\prod_{v\in
  \Places_k}(1-q_v^{-s})/(1-q_v^{1-s})$,
where $\zeta_k$ is the Dedekind zeta function of $k$, and he conjectured the
same formula  to hold for arbitrary number fields $k$.
This was proved by Stasinski and Voll as a very special case of
\cite[Thm~B]{SV14}.
Both approaches proceed by computing the corresponding local zeta functions
$\zeta_{\mathsf H(\fo_v)}(s)$ as indicated by the Euler product above.

Corollary~\ref{cor:rigidity} shows that, disregarding a finite number of
exceptional places, such a regular behaviour of local
representation zeta functions under base extension is a general phenomenon.
Indeed, knowing that $\zeta_{\mathsf H(\ZZ_p)}(s) = (1-p^{-s})/(1-p^{1-s})$ for
(almost) all primes~$p$, 
using Corollary~\ref{cor:rigidity} and the (deep) fact that
the $\Zeta^{\GG,\wirr}$ defined in \S\ref{sss:unipotent} are of Denef type,
we can immediately deduce that $\zeta_{\mathsf H(\fO_K)}(s) =
(1-q_K^{-s})/(1-q_K^{1-s})$ for almost all $p$ and all finite extensions
$K/\QQ_p$.
We note that the exclusion of finitely many places, unnecessary as it might be
in this particular case, is deeply ingrained in the techniques underpinning
the proof of Theorem~\ref{thm:rigidity} (but see Lemma~\ref{lem:good_reduction}).
\end{ex}

In analogy with \cite[\S 1.2.4]{dSW08}, we say that a 
$k$-local map $\Zeta$ is \emph{uniform} if it is equivalent to
$(v,f)\mapsto W(q_v^f,Y_1,\dotsc,Y_m)$ for a rational function $W\in
\QQ(X,Y_1,\dotsc,Y_m)$ which is regular at each point $(q,Y_1,\dotsc,Y_m)$ for
each integer $q > 1$.
We then say that $W$ \emph{uniformly represents}~$\Zeta$.
For a more literal analogue of the notion of uniformity in~\cite{dSW08} (which
goes back to \cite[\S 5]{GSS88}), we should only insist that $\Zeta(v,1) =
W(q_v,Y_1,\dotsc,Y_m)$ for almost all $v\in \Places_k$.
However, these two notions of uniformity actually coincide:

\begin{cor}
  \label{cor:uniform}
  Let $\Zeta$ be a $k$-local map of Denef type.
  Let $P\subset \Places_k\setminus S_\Zeta$ have density~$1$ and
  let $W\in \QQ(X,Y_1,\dotsc,Y_m)$ be regular at $(q,Y_1,\dotsc,Y_m)$ for all
  integers $q>1$.
  If $\Zeta(v,1) = W(q_v,Y_1,\dotsc,Y_m)$ for all $v\in P$,
  then $\Zeta$ is uniformly represented by $W$.
\end{cor}
\begin{proof}
  Apply Corollary~\ref{cor:rigidity} with $\Zeta'(v,f) = W(q_v^f,Y_1,\dotsc,Y_m)$.
\end{proof}

This in particular applies to a large number of examples of
uniform $p$-adic subalgebra and ideal zeta functions computed by Woodward
\cite{dSW08}.
Specifically, \cite[Ch.~2]{dSW08} contains numerous examples
of Lie rings $\mathsf L$ such that $\zeta^{\le}_{\mathsf L\otimes \ZZ_p}(s) =
W(p,p^{-s})$ for some $W(X,Y) \in \QQ(X,Y)$ and all rational
primes $p$ (or almost all of them);
the denominators of the rational functions $W(X,Y)$ are products of factors $1-X^aY^b$.
Corollary~\ref{cor:uniform} now shows that for almost all $p$ and all finite extensions
$K/\QQ_p$, we necessarily have
$\zeta^{\le}_{\mathsf L\otimes \fO_K}(s) =
W(q_K^{\phantom s},q_K^{-s})$. 
We note that since Woodward's computations of zeta functions are based on $p$-adic
integration, we expect them to immediately extend to finite extensions of
$\QQ_p$, thus rendering this application of Corollary~\ref{cor:uniform} unnecessary.
However, as Woodward provides few details on his computations
(see~\cite[Ch.~2]{Woo05}), the present author is unable to confirm this. 
It is perhaps a consequence of the group-theoretic origin of the subject
(see the final paragraph of~\S\ref{s:local_zetas}) 
that local base extensions, in the sense considered here, have
not been investigated for subalgebra and ideal zeta functions prior to~\cite{topzeta}.

%%%%%%%%%%%%%%%%%%%%%%%%%%%%%%%%%%%%%%%%%%%%%%%%%%%%%%%%%%%%%%%%%%%%%%%%
\subsection{Proof of Theorem~\ref{thm:rigidity}}
\label{ss:pf_rigidity}
%%%%%%%%%%%%%%%%%%%%%%%%%%%%%%%%%%%%%%%%%%%%%%%%%%%%%%%%%%%%%%%%%%%%%%%%

We first recall some facts on the subject of counting points on varieties
from \cite[Ch.~3--4]{Ser12}.
Let $G$ be a profinite group.
We let $\Cl(g)$ denote the conjugacy class of $g\in G$
and write $\Cl(G) := \{ \Cl(g) : g\in G\}$.
We may naturally regard $\Cl(G)$ as a profinite space by 
endowing it with the quotient topology or, equivalently, by
identifying $\Cl(G) = \varprojlim_{N\normal_o G} \Cl(G/N)$, where $N$ ranges
over the open normal subgroups of $G$ and each of the finite sets $\Cl(G/N)$ is
regarded as a discrete space.

As before, let $k$ be a number field with ring of integers $\fo$.
Fix an algebraic closure $\bar k$ of~$k$.
For a finite set $S\subset \Places_k$, 
let $\Gamma_S$ denote the Galois group of the
maximal extension of $k$ within $\bar k$ which is unramified outside of $S$.
For $v\in \Places_k\setminus S$, choose an element
$g_v \in \Gamma_S$ in the conjugacy class of geometric Frobenius
elements associated with $v$, see \cite[\S\S 4.4, 4.8.2]{Ser12}.
The following is a consequence of Chebotarev's density theorem.
\begin{thm}[{Cf.~\cite[Thm~6.7]{Ser12}}]
  \label{thm:chebotarev_dense}
  Let $P\subset \Places_k\setminus S$ have density $1$.
  Then $\{ \Cl(g_v) : v\in P\}$ is a dense subset of $\Cl(\Gamma_S)$.
\end{thm}

Fix an arbitrary rational prime $\ell$.
Recall that a \emph{virtual $\ell$-adic character} of a profinite group
$G$ is a map $\alpha\colon G \to \QQ_\ell$ of the form $g \mapsto \sum_{i=1}^u
c_i \trace(\varrho_i(g))$, where $c_1,\dotsc,c_u\in \ZZ$ and each
$\varrho_i$ is a continuous homomorphism from $G$ into some
$\GL_{n_i}(\QQ_\ell)$.
Note that such a map $\alpha$ induces a continuous map $\Cl(G)\to \QQ_\ell$.
The set of virtual $\ell$-adic characters of $G$ forms a commutative ring
under pointwise operations.

The next result follows from
Grothendieck's trace formula \cite[Thm~4.2]{Ser12}
and the generic cohomological behaviour
of reduction modulo non-zero primes of $\fo$ \cite[Thm~4.13]{Ser12}.
\begin{thm}[{See \cite[Ch.~4]{Ser12} and cf.~\cite[\S\S 6.1.1--6.1.2]{Ser12}}]
  \label{thm:trace}
  Let $V$ be a separated $\fo$-scheme of finite type.
  Then there exist a finite set $S \subset \Places_k$ and a virtual $\ell$-adic character
  $\alpha$ of $\Gamma_S$ such that
  $\noof{V(\RF_v^{(f)})} = \alpha(g_v^f)$ for all $v\in \Places_k\setminus S$
  and all $f \in \NN$.
\end{thm}

Note that since $\noof{V(\RF_v)}\in \NN\cup\{0\}$ for all $v\in \Places_k$ and
$\{ \Cl(g_v) :v\in \Places_k\setminus S\}$ is dense in $\Cl(\Gamma_S)$,
the virtual character $\alpha$ in Theorem~\ref{thm:trace} is necessarily $\ZZ_\ell$-valued.

\begin{proof}[Proof of Theorem~\ref{thm:rigidity}]
  There exists a non-zero $D \in \ZZ[X,Y_1,\dotsc,Y_m]$ such that
  $D W_i\in \ZZ[X,Y_1,\dotsc,Y_m]$ for $i=1,\dotsc,r$
  and $D(q,Y_1,\dotsc,Y_m)\not= 0$ for integers $q > 1$.
  The proof of Theorem~\ref{thm:rigidity} is thus reduced to the case
  $W_1,\dotsc,W_r\in \ZZ[X,Y_1,\dotsc,Y_m]$.
  By considering the coefficients of each monomial $Y_1^{a_1}\dotsb
  Y_m^{a_m}$, we may assume that $W_1,\dotsc,W_r\in \ZZ[X]$.

  Let $V_0 := \AA^1_{\fo}$.
  By Theorem~\ref{thm:trace}, there exists a finite set $S\subset \Places_k$ and
  for $0\le i\le r$, a continuous virtual $\ell$-adic character
  $\gamma_i\colon \Gamma_S \to \ZZ_\ell$ with
  $\noof{V_i(\RF_{v}^{(f)})} = \gamma_i(g_v^f)$ 
  for all $v\in \Places_k\setminus S$ and $f\in \NN$.
  By construction, the virtual character $\alpha := \sum_{i=1}^r \gamma_i \dtimes
  W_i(\gamma_0)\colon \Gamma_{S} \to \ZZ_\ell$ then satisfies
  $\alpha(g_v) = 0$ for $v \in P\setminus S$ whence $\alpha = 0$ by
  Theorem~\ref{thm:chebotarev_dense}.
  The theorem follows by evaluating $\alpha$ at the $g_v^f$.
\end{proof}

\begin{lemma}
  \label{lem:good_reduction}
  In the setting of Theorem~\ref{thm:rigidity},
  suppose that each $V_i\otimes_{\fo} k$ is smooth and proper over $k$.
  Let $S'\subset \Places_k$ such that $v\in\Places_k\setminus S'$ if and only if
  $V_i\otimes \RF_v$ is smooth and proper over $\RF_v$ for each $i=1,\dotsc,r$.
  Then we may take $S = S'$ in Theorem~\ref{thm:rigidity}.
\end{lemma}
\begin{proof}
  By the proof of Theorem~\ref{thm:rigidity} and \cite[\S 4.8.4]{Ser12},
  for every rational prime $\ell$, we may take
  $S = S'\cup\{ v\in \Places_k :  p_v = \ell \}$ in Theorem~\ref{thm:rigidity}.
  The claim follows since $\ell$ is arbitrary.
\end{proof}

%%%%%%%%%%%%%%%%%%%%%%%%%%%%%%%%%%%%%%%%%%%%%%%%%%%%%%%%%%%%%%%%%%%%%%%%
\section{Application: topological zeta functions and \texorpdfstring{$p$}{p}-adic formulae}
\label{s:top}
%%%%%%%%%%%%%%%%%%%%%%%%%%%%%%%%%%%%%%%%%%%%%%%%%%%%%%%%%%%%%%%%%%%%%%%%

Denef and Loeser introduced topological zeta functions as singularity invariants
associated with polynomials.
These zeta functions were first obtained arithmetically using
an $\ell$-adic limit ``$p \to 1$'' of Igusa's $p$-adic zeta functions, see \cite{DL92}.
For a geometric approach using motivic integration, see \cite{DL98}.
Based on the motivic point of view, du~Sautoy and Loeser introduced topological
subalgebra zeta functions. 
In this section, by combining the arithmetic point of view and
Theorem~\ref{thm:rigidity}, we show that $p$-adic formulae alone determine
associated topological zeta functions, and we discuss consequences.

We first recall the formalism from \cite[\S 5]{topzeta} in the version from \cite[\S 3.1]{unipotent}.
Thus, for $e \in \QQ[s_1,\dotsc,s_m]$, the expansion
$
X^e :=
\sum_{d=0}^\infty \binom e d (X-1)^d
\in \QQ[s_1,\dotsc,s_m]\llb X-1\rrb
$
gives rise to an embedding $h \mapsto h(X,X^{-s_1},\dotsc,X^{-s_m})$
of $\QQ(X,Y_1,\dotsc,Y_m)$ into the field $\QQ(s_1,\dotsc,s_m)\llp X-1\rrp$.
Let $\MM[X,Y_1,\dotsc,Y_m]$ denote the subalgebra of $\QQ(X,Y_1,\dotsc,Y_m)$
consisting of those $W = g/h$ with
$W(X,X^{-s_1},\dotsc,X^{-s_m})\in \QQ(s_1,\dotsc,s_m)\llb X-1\rrb$,
where $g\in \QQ[X^{\pm 1},Y_1^{\pm 1},\dotsc,Y_m^{\pm 1}]$ and $h$ is a finite
product of non-zero ``cyclotomic factors''
of the form $1-X^a Y_1^{b_1}\dotsb Y_m^{b_m}$ for
$a,b_1,\dotsc,b_m\in \ZZ$.
Given $W\in \MM[X,Y_1,\dotsc,Y_m]$, write $$\red W :=
W(X,X^{-s_1},\dotsc,X^{-s_m}) \bmod {(X-1)} \in \QQ(s_1,\dotsc,s_m).$$

We say that a $k$-local map $\Zeta$ in $m$ variables is \emph{expandable} if it is
equivalent to a finite sum of maps 
of the form $[V\dtimes W]$, where $V$ is a separated $\fo$-scheme of finite type and $W\in
\MM[X,Y_1,\dotsc,Y_m]$;
we chose the term ``expandable'' to indicate that each $W$ admits a symbolic
power series expansion in $X-1$.

\begin{ex}
  \label{ex:expandable}
  \quad
  \begin{enumerate}
  \item
    The local maps of the form
    $\Zeta^{a_1,\dotsc,a_m}$, $\Zeta^{\GG,\wirr}$, and $\Zeta^{\bm\fg,\irr}$  from
    \S\ref{ss:denef_examples} are expandable,
    see \cite[\S\S 3.2--3.3]{unipotent} and
    cf.~\cite[Ex.~5.11(i)--(ii)]{topzeta}.
  \item
    \label{ex:expandable2}
    If $\cA$ and $M$ from \S\ref{sss:subobjects} have $k$-dimension $d$, then
    the local maps 
    $(1-X^{-1})^d \Zeta^{\cA}$ and $(1-X^{-1})^d \Zeta^{\cE\acts M}$
    (pointwise products) are both expandable by \cite[Thm~5.16]{topzeta}.
    
    The factor $(1-X^{-1})^d$ is included to ensure expandability.
    For example, if the multiplication $\cA\otimes_k\cA \to \cA$ is zero,
    then $\Zeta^{\cA}$ is uniformly represented by $W(X,Y) :=
    1/((1-Y)(1-XY)\dotsb(1-X^{d-1}Y))$;
    this reflects the well-known fact that the zeta function enumerating the open
    subgroups of $\ZZ_p^d$ is $1/\prod_{i=0}^{d-1}(1-p^{i-s})$.
    Hence,
    \[
    W(X,X^{-s}) = \frac{1}{s(s-1)\dotsb(s-(d-1))} (X-1)^{-d} + \dotsb \in
    \QQ(s)\llp X-1\rrp
    \]
    which is not a power series in $X-1$ for $d > 0$.
    In contrast, $W' := (1-X^{-1})^d W(X,Y)$ belongs to $\MM[X,Y]$ and $\red{W'} = 1/(s(s-1)\dotsb(s-(d-1)))$.
  \end{enumerate}
\end{ex}

If $V$ is a $k$-variety, then any embedding $k \incl \CC$ allows us to regard
$V(\CC)$ as a $\CC$-analytic space.
Cohomological comparison theorems (see e.g.~\cite{Kat94}) show that the
topological Euler characteristic $\Euler(V(\CC))$ does not depend on the chosen embedding $k \incl \CC$.
The following formalises fundamental insights from \cite{DL92}.

\begin{thm}[{Cf.~\cite[Thm~5.12]{topzeta}}]
  \label{thm:pre_rigidity}
  Let $V_1,\dotsc,V_r$ be separated $\fo$-schemes of finite type,
  let $W_1,\dotsc,W_r\in \MM[X,Y_1,\dotsc,Y_m]$, and let $S\subset \Places_k$ be finite.
  Suppose that $$\sum\limits_{i=1}^r \noof{V_i(\RF_v^{(f)})} \dtimes W_i(q_v^f,Y_1,\dotsc,Y_m) = 0$$
  for all $v\in \Places_k\setminus S$  and $f\in \NN$.
  Then $$\sum\limits_{i=1}^r \Euler(V_i(\CC)) \dtimes \red{W_i} = 0.$$
\end{thm}

Let $\Zeta$ be an expandable $k$-local map so that
$\Zeta$ is equivalent to a sum
$[V_1 \dtimes W_1] + \dotsb + [V_r \dtimes W_r]$, where $V_1,\dotsc,V_r$ are
separated $\fo$-schemes of finite type and $W_1,\dotsc,W_r\in
\MM[X,Y_1,\dotsc,Y_m]$.
By Theorem~\ref{thm:pre_rigidity}, we may unambiguously define the
\emph{topological zeta function} $\Zeta_{\topo} \in \QQ(s_1,\dotsc,s_m)$
associated with $\Zeta$ via
\[
\Zeta_{\topo} := \sum_{i=1}^r \Euler(V_i(\CC))\dtimes \red{W_i}.
\]
By applying this definition to the expandable local maps in
Example~\ref{ex:expandable}, topological versions of the zeta
functions from \S\ref{ss:denef_examples} are defined;
the topological zeta function associated with a polynomial $f\in
k[X_1,\dotsc,X_n]$ of Denef and Loeser is of course obtained as a special case;
the same is true of the topological subalgebra
zeta functions of du~Sautoy and Loeser, see \cite[Rk~5.18]{topzeta}.

\begin{cor}
  \label{cor:top_rigidity}
  Let $\Zeta$ and $\Zeta'$ be $k$-local maps of Denef type.
  Let $P\subset \Places_k\setminus(S_{\Zeta}\cup S_{\Zeta'})$ have density $1$
  and suppose that $\Zeta(v,1) = \Zeta'(v,1)$ for all $v\in P$.
  If $\Zeta$ is expandable, then so is $\Zeta'$ and
  $\Zeta^{\phantom\prime}_{\topo} = \Zeta'_{\topo}$.
\end{cor}
\begin{proof}
  Combine Corollary~\ref{cor:rigidity} and Theorem~\ref{thm:pre_rigidity}.
\end{proof}

For instance, given $\ZZ$-forms $\A$ and $\mathsf B$ of two $\QQ$-algebras
$\cA$ and $\mathcal B$ of the same dimension, if $\zeta^{\le}_{\A\otimes \ZZ_p}(s) = \zeta^{\le}_{\mathsf B\otimes \ZZ_p}(s)$ for almost all $p$, then the topological
subalgebra zeta functions of $\A$ and $\mathsf B$ necessarily coincide;
we note that the assumption on the dimensions could be dropped if
\cite[Conj.~I]{topzeta} was known to true.
This fact seems to suggest an advantage of the original arithmetic definition
of topological zeta functions via explicit $p$-adic formulae compared with 
specialisations of motivic zeta functions
as in~\cite{DL98}---indeed, it is not presently known if $p$-adic identities
imply an equality of associated motivic zeta functions, see 
\cite[\S\S 7.11--7.14]{dSL04}.
It is therefore presently conceivable
that $\Zeta^{\cA}$ and $\Zeta^{\mathcal B}$ might be equivalent even though the
motivic subalgebra zeta functions of $\cA$ and $\mathcal B$ differ.

In the introductions to his previous articles on the subject
\cites{topzeta,topzeta2,unipotent}, the author informally ``read off''
topological zeta functions from $p$-adic ones as the constant term as a
series in $p-1$ (taking into account correction factors as in
Example~\ref{ex:expandable}(\ref{ex:expandable2})).
The informal nature was due to some of the $p$-adic formulae used, in particular
those from \cite{dSW08}, only being known under variation of $p$ but not under
base extension, see the final paragraph of \S\ref{ss:main}.
By combining Corollary~\ref{cor:top_rigidity} and the following technical lemma, we may
now conclude that this informal approach for deducing topological subobject
zeta functions from suitably uniform $p$-adic formulae is fully rigorous.
In particular, Woodward's many examples of uniform subalgebra and ideal zeta
functions immediately provide us with knowledge of the corresponding topological
zeta functions.

\begin{lemma}
  \label{lem:auto_M}
  Let $W = g(X,Y_1,\dotsc,Y_m) /\prod_{i\in I}(1-X^{a_i}Y_1^{b_{i1}}\dotsb
  Y_m^{b_{im}}) \in \QQ(X,Y_1,\dotsc,Y_m)$
  for a Laurent polynomial $g \in \QQ[X^{\pm 1},Y_1^{\pm 1},\dotsc,Y_m^{\pm 1}]$,
  a finite set $I$, and integers $a_i,b_{ij}$ with
  $(a_i,b_{i1},\dotsc,b_{im})\not = 0$ for $i\in I$.
  Let $\Zeta$ be an expandable $k$-local map
  which is uniformly represented by $W$.
  Then $W\in \MM[X,Y_1,\dotsc,Y_m]$ and thus $\Zeta_{\topo} = \red W$.
\end{lemma}
\begin{proof}
  By \cite[Thm~5.12]{topzeta}, there exists a finite union of affine
  hyperplanes $\cH \subset\AA^m_{\ZZ}$ 
  such that for any rational prime $\ell$ and almost all $v\in \Places_k$,
  there exists $d\in \NN$ such that 
  \[
  \NN\times \ZZ^m\setminus \cH(\ZZ)\to \QQ, 
  \quad
  (f;\bm s) \mapsto \hat\Zeta_{k_v^{(df)}}(\bm s_1,\dotsc,\bm s_m)
  \]
  is well-defined and admits a continuous extension
  $\Phi\colon\ZZ_\ell^{\phantom m}\!\!\times
  \ZZ_\ell^m \setminus \cH(\ZZ_\ell) \to \QQ_\ell$.
  By enlarging~$\cH$, we may assume that
  $a_i \not= \sum_{j=1}^m b_{ij}\bm s_j$ for $i\in I$ 
  and $\bm s \in \ZZ_\ell^m\setminus\cH(\ZZ_\ell)$.
  Since $W$ uniformly represents $\Zeta$, 
  we may choose $(\cH,\ell,v,d)$ such that,
  in addition to the conditions from above,
  $\Phi(f;\bm s) = W(q_v^{df},q_v^{-df\bm s_1},\dotsc,q_v^{-df\bm s_m})$
  for $(f;\bm s)\in \NN\times\ZZ^m\setminus\cH(\ZZ)$.
  Finally, we may also assume that $\ell\not= 2$ and that $q_v^d \equiv 1 \bmod {\ell}$.

  Let $w$ be the $(X-1)$-adic valuation of $g(X,X^{-s_1},\dotsc,X^{-s_m}) \in
  \QQ[s_1,\dotsc,s_m]\llb X-1\rrb$.
  Define $G(s_1,\dotsc,s_m;X-1)\in \QQ[s_1,\dotsc,s_m]\llb X-1\rrb$ by
  \[
  g(X,X^{-s_1},\dotsc,X^{-s_m}) = (X-1)^w \dtimes G(s_1,\dotsc,s_m;X-1)
  \]
  and note that the constant term $G(s_1,\dotsc,s_m;0)$ is non-zero.
  As in the proof of \cite[Lem.~5.6]{topzeta},
  using the $\ell$-adic binomial series, 
  for $f\in \ZZ_\ell$ and $\bm s\in \ZZ_\ell^m$, we have
  \begin{equation}
    \label{eq:gG}
    g( (q_v^d)^f,(q_v^d)^{-f \bm s_1},\dotsc,(q_v^d)^{-f\bm s_m})
    = ((q_v^d)^{f}-1)^w \dtimes G(\bm s;(q_v^d)^f-1).
  \end{equation}
  Choose $\bm s_\infty\in \ZZ_\ell^m\setminus \cH(\ZZ_\ell)$ such that
  $G(s_1,\dotsc,s_m;0)$ does not vanish at $\bm s_\infty$.
  Further choose $(f_n)_{n\in \NN}\subset\NN$ with $f_n \to 0$ in $\ZZ_\ell$
  and $(\bm s_n)_{n\in \NN}\subset\ZZ^m\setminus \cH(\ZZ)$ with $\bm s_n\to
  \bm s_\infty$ in $\ZZ_\ell^m$.
  Define $x_n := q_v^{df_n}$.
  Then
  \begin{equation}
    \label{eq:WPhi}
    W(x_n^{\phantom{s_n}}\!\!,x_n^{-\bm s_{n1}},\dotsc,x_n^{-\bm s_{nm}}) =
    \Phi(f_n;\bm s_n) \to
    \Phi(0;\bm s_\infty)\in \QQ_\ell
  \end{equation}
  as $n\to \infty$.
  Let $e_{in} := a_i - b_{i1} \bm s_{n1} - \dotsb - b_{im} \bm s_{nm}$.
  Using~\eqref{eq:gG}, the left-hand side of \eqref{eq:WPhi}
  coincides with
  \[
  G(\bm s_n;x_n-1) \dtimes \frac{(x_n-1)^w}{\prod_{i\in
      I}\bigl( 1-x_n^{e_{in}}\bigr)}
  \]
  and since the left factor converges to $G(\bm s_\infty;0) \in \QQ_\ell^\times$
  for $n\to \infty$,
  the right one converges to an element of $\QQ_\ell$.
  Since $\abs{x_n^e-1}_\ell = \abs{e}_\ell\dtimes \abs{x_n-1}_\ell$ for $e\in
  \ZZ_\ell$
  and $\lim_{n\to\infty}e_{in}\not= 0$ for $i\in I$,
  this is easily seen to imply $w \ge \noof I$ whence $W\in \MM[X,Y_1,\dotsc,Y_m]$.
\end{proof}

%%%%%%%%%%%%%%%%%%%%%%%%%%%%%%%%%%%%%%%%%%%%%%%%%%%%%%%%%%%%%%%%%%%%%%%%
\section{Application: functional equations and uniformity}
\label{s:feqn}
%%%%%%%%%%%%%%%%%%%%%%%%%%%%%%%%%%%%%%%%%%%%%%%%%%%%%%%%%%%%%%%%%%%%%%%%

Let $\mathsf A$ be a possibly non-associative $\ZZ$-algebra whose underlying
$\ZZ$-module is finitely generated of torsion-free rank~$d$. 
Voll~\cite[Thm~A]{Vol10} proved that, for almost all primes $p$, the local
subalgebra zeta function $\zeta_{\mathsf A\otimes \ZZ_p}^{\le}(s)$ satisfies the functional
equation
\begin{equation}
  \label{eq:proto_feqn}
  \zeta^{\le}_{\mathsf A\otimes \ZZ_p}(s) \big \vert_{p \to p^{-1}} = (-1)^d
  p^{\binom{d}2-ds} \dtimes
  \zeta^{\le}_{\mathsf A\otimes \ZZ_p}(s),
\end{equation}
where the operation of ``inverting $p$'' is defined with
respect to a judiciously chosen formula~\eqref{eq:denef}.
One may ask to what extent the operation ``$p \to p^{-1}$'' depends on the
chosen $V_i$ and $W_i$ in \eqref{eq:denef}.
As we will see in this section, it does not, in the sense that knowing that
the functional equation established by Voll behaves well under local base
extensions (as in \cite{DM91} and \cite{AKOV13}), we can conclude that any other
formula~\eqref{eq:denef} (subject to minor technical constraints) behaves in the same way
under inversion of $p$ for almost all~$p$.
In particular, we will see that the left-hand side of \eqref{eq:proto_feqn}
takes the expected form of symbolically inverting $p$ for uniform examples, a
fact which does not seem to have been spelled out before.

We begin by recalling the formalism for ``inverting primes'' from \cite{DM91}
which we then combine with our language of local maps from \S\ref{s:local_maps}.
First, let $U$ be a separated scheme of finite type over a finite field $\FF_q$.
As explained in \cite[\S 2]{DM91} and \cite[\S 1.5]{Ser12},
using the rationality of the Weil zeta function of $U$,
there are non-zero $m_1,\dotsc,m_u\in \ZZ$ and distinct non-zero
$\alpha_1,\dotsc,\alpha_u\in \bar\QQ$ such that for each $f\in \NN$, we have
\begin{equation}
  \label{eq:count}
  \noof{U(\FF_{q^f})} = \sum_{i=1}^u m_i \alpha_i^f.
\end{equation}
By \cite[Lem.~2]{DM91}, the $(m_i,\alpha_i)$ are unique up to permutation.
As in~\cite[\S 1.5]{Ser12}, one may thus use \eqref{eq:count} to
unambiguously extend the definition of $\noof{U(\FF_{q^f})}$ to arbitrary $f\in \ZZ$.
Note that by considering the action of $\mathrm{Gal}(\bar\QQ/\QQ)$ on the right-hand side of \eqref{eq:count},
the uniqueness of $\{(m_1,\alpha_1),\dotsc,(m_u,\alpha_u)\}$ implies
that $\noof{U(\FF_{q^f})}\in \QQ$ for $f\in \ZZ$.

\begin{lemma}
  \label{lem:local_rigidity}
  Let $U_1,\dotsc,U_r$ be separated $\FF_q$-schemes of finite type.
  Let $W_1,\dotsc,W_r\in \QQ(X,Y_1,\dotsc,Y_m)$ each be regular at
  $(q^f,Y_1,\dotsc,Y_m)$ for all $f\in \ZZ\setminus\{0\}$.
  If $$\sum_{i=1}^r \noof{U_i(\FF_{q^f})}\dtimes W_i(q^f,Y_1,\dotsc,Y_m) = 0$$
  for all $f \in \NN$, then this identity extends to all $f\in \ZZ\setminus\{0\}$.
\end{lemma}
\begin{proof}
  As in the proof of Theorem~\ref{thm:rigidity}, we may reduce to the case where
  each $W_i\in \ZZ[X]$. The result then follows from
  \cite[Lem.~2]{DM91} and its Corollary, cf.~the proof of \cite[Lem.~3]{DM91}.
\end{proof}

\begin{rem}
  In the cases of interest to us, the denominators of the $W_i$ might be
  divisible by polynomials of the form $1-X^e$. This justifies
  the exclusion of $f = 0$ in Lemma~\ref{lem:local_rigidity}.
\end{rem}

\begin{cor}
  \label{cor:invrigid}
  Let $V_1,\dotsc,V_r$ be separated $\fo$-schemes of finite type and
  let $W_1,\dotsc,W_r \in \QQ(X,Y_1,\dotsc,Y_m)$ each be regular at
  $(q^f,Y_1,\dotsc,Y_m)$ for all integers $q > 1$ and $f\in \ZZ\setminus\{0\}$.
  Let $P\subset \Places_k$ have density $1$ 
  and suppose that for all $v \in P$,
 $$\sum\limits_{i=1}^r \noof{V_i(\RF_v)} \dtimes W_i(q_v,Y_1,\dotsc,Y_m) = 0.$$
  Then there exists a finite set $S\subset \Places_k$ such that for all $v \in
  \Places_k\setminus S$ and all $f\in \ZZ\setminus\{0\}$,
  $$\sum\limits_{i=1}^r \noof{V_i(\RF_v^{(f)})} \dtimes W_i(q_v^f,Y_1,\dotsc,Y_m) = 0.$$
\end{cor}
\begin{proof}
  Combine Theorem~\ref{thm:rigidity} and Lemma~\ref{lem:local_rigidity}.
\end{proof}

Let $V_1,\dotsc,V_r$ be separated $\fo$-schemes of finite type, $W_1,\dotsc,W_r\in
\QQ(X,Y_1,\dotsc,Y_m)$ each be regular at $(q^f,Y_1,\dotsc,Y_m)$ for all
integers $q > 1$ and $f\in \ZZ\setminus\{0\}$, and let $\Zeta$ be a $k$-local map which is equivalent to the sum
$[V_1\dtimes W_1] + \dotsb + [V_r\dtimes W_r]$.
For a sufficiently large finite set $S\subset\Places_k$, the map
\begin{equation}
  \label{eq:ext}
\Zeta_*\colon
\Places_k\setminus S \times \ZZ\setminus\{0\} \to
\QQ(Y_1,\dotsc,Y_m),
\,\,
(v,f) \mapsto \sum_{i=1}^r \noof{V_i(\RF_v^{(f)})}\dtimes  W_i(q_v^f,Y_1^{\sign(f)},\dotsc,Y_m^{\sign(f)})
\end{equation}
satisfies $\Zeta_*(v,f) = \Zeta(v,f)$ for $v\in \Places_k\setminus S$ and $f\in \NN$.
By Lemma~\ref{lem:local_rigidity},
up to enlarging~$S$, the map $\Zeta_*$ is uniquely determined by the equivalence
class of $\Zeta$ and therefore, in particular, independent of the chosen $V_i$ and $W_i$.
In accordance with the definition of $\hat\Zeta$ from~\S\ref{ss:denef_type},
we write $$\hat\Zeta_*(v,f) := \Zeta_*(v,f)(q_v^{-fs_1},\dotsc,q_v^{-fs_m}).$$

\begin{lemma}[{Cf.~\cite[Cor.\ to Thm~4]{DM91}}]
  \label{lem:uninverse}
  Suppose that the $k$-local map $\Zeta$ is uniformly represented by a rational
  function $W\in \QQ(X,Y_1,\dotsc,Y_m)$ which is regular at each point
  $(q^f,Y_1,\dotsc,Y_m)$ for all integers $q>1$ and $f\in \ZZ\setminus\{0\}$.
  Then for almost all $v\in \Places_k$ and all $f\in \NN$, we have
  $\Zeta_*(v,-f) = W(q_v^{-f},Y_1^{-1},\dotsc,Y_m^{-1})$
  and therefore
  $\hat\Zeta_*(v,-f) = W(q_v^{-f},q_v^{fs_1},\dotsc,q_v^{fs_m})$. \qed
\end{lemma}

The known explicit formulae for the zeta functions from
\S\S\ref{sss:subobjects}--\ref{sss:perfect} 
satisfy the regularity conditions on the $W_i$ from above, allowing us to
consider extensions of the form~\eqref{eq:ext}.
These extensions satisfy the following functional equations.

\begin{thm}
  \label{thm:functional}
  \quad
  \begin{enumerate}
    \item
      \label{thm:functional1}
      \textup{(\cite[Thm~A]{Vol10})}
      Let $\cA$ be a not necessarily associative $k$-algebra of dimension $d$.
      Then for almost all $v\in \Places_k$ and all $f\in \NN$,
      $$\Zeta_*^{\cA}(v,-f) = (-1)^d (q_v^f)^{\binom d 2}Y^d \dtimes \Zeta^{\cA}(v,f).$$
    \item
      \textup{(\cite[Thm~A]{AKOV13})}
      Let $\bm\fg$ be a perfect Lie $k$-algebra of dimension $d$.
      Then for almost all $v\in \Places_k$ and all $f\in \NN$,
      $$\Zeta_*^{\bm\fg,\irr}(v,-f) = q_v^{df} \dtimes \Zeta^{\bm\fg,\irr}(v,f).$$
    \item
      \textup{(\cite[Thm~A]{SV14})}
      Let $\GG$ be a unipotent algebraic group over $k$.
      Let $d$ be the dimension of the (algebraic) derived subgroup of $\GG$.
      Then for almost all $v\in \Places_k$ and all $f\in \NN$,
      $$\Zeta_*^{\GG,\wirr}(v,-f) = q_v^{df} \dtimes \Zeta^{\GG,\wirr}(v,f).$$
    \end{enumerate}
\end{thm}

\begin{rem}
  \quad
  \begin{enumerate}
    \item
      Regarding Theorem~\ref{thm:functional}(\ref{thm:functional1}), we
      note that while \cite[Thm~A]{Vol10} only spells out the functional equation for $k = \QQ$ and $f =
      1$, the general case follows using the arguments in~\cite[\S 4.2]{AKOV13} and
      the fact that the formalism from~\cite[\S 3.1]{Vol10} applies to extensions of
      $\QQ$ and $\QQ_p$, respectively, after applying simple textual changes such as
      replacing $p$ by either $q$ or a fixed uniformiser.
    \item
      The functional equations in \cite[Thm~B--C]{Vol10} can be similarly
      rephrased using the formalism from the present article.
    \end{enumerate}
\end{rem}

Let us return to the setting from the opening of this section.
Suppose that \eqref{eq:denef} holds for almost all~$p$, where the $V_i$ are
separated $\ZZ$-schemes of finite type and the $W_i\in \QQ(X,Y)$ are regular at
$(q^f,Y)$ for all integers $q > 1$ and $f\in \ZZ\setminus\{0\}$.
Then Corollary~\ref{cor:invrigid} and
Theorem~\ref{thm:functional}(\ref{thm:functional1}) show that for almost all $p$,
\[
\sum_{i=1}^r \noof{V_i(\FF_{p^{-1}})}\dtimes W_i(p^{-1},p^s) =
(-1)^dp^{\binom d 2 - ds} \dtimes \zeta^\le_{\mathsf A\otimes\ZZ_p}(s),
\]
regardless of whether the $V_i$ and $W_i$ were obtained 
as in the proof of Theorem~\ref{thm:functional}(\ref{thm:functional1}) or
not.
Note that in the uniform case ($r = 1$, $V_1 = \Spec(\ZZ)$), we 
obtain
\begin{equation}
  \label{eq:unifun}
  W_1(X^{-1},Y^{-1}) = X^{\binom d 2}Y^d \dtimes W_1(X,Y);
\end{equation}
recall that uniformity in the sense of \cite{dSW08} essentially coincides with
our notion of uniformity for local maps thanks to Corollary~\ref{cor:uniform}.
This explains why the various uniform examples of local subalgebra 
zeta functions computed by Woodward \cite{Woo05,dSW08}
satisfy the functional equation~\eqref{eq:unifun} even though the methods he
used to compute them differ considerably from Voll's approach.

%%%%%%%%%%%%%%%%%%%%%%%%%%%%%%%%%%%%%%%%%%%%%%%%%%%%%%%%%%%%%%%%%%%%%%%%
% References
%%%%%%%%%%%%%%%%%%%%%%%%%%%%%%%%%%%%%%%%%%%%%%%%%%%%%%%%%%%%%%%%%%%%%%%%
{
  \bibliographystyle{abbrv}
  \footnotesize
  \bibliography{stability}
}

\end{document}